\documentclass[12pt, a4paper]{article}

\usepackage{datetime}

\usepackage{color}

\usepackage{eucal}

\usepackage[all]{xy}

\usepackage{tikz}

\usepackage{amssymb, amsmath, amsthm, eucal}

\usepackage{mathptmx}
\usepackage[scaled=.90]{helvet}
\usepackage{courier}

\usepackage[pdftex]{hyperref}

\hypersetup{
    colorlinks = true,
    linkcolor = black,
    anchorcolor = black,
    citecolor = black,
    filecolor = blue,
    urlcolor = blue
}

\newcommand{\bbF}{\mathbb{F}}

\newcommand{\bbQ}{\mathbb{Q}}
\newcommand{\bbR}{\mathbb{R}}

\newcommand{\bbZ}{\mathbb{Z}}

\newcommand{\calF}{\mathcal{F}}

\newcommand{\rmE}{\mathrm{E}}
\newcommand{\rmF}{\mathrm{F}}
\newcommand{\rmG}{\mathrm{G}}
\newcommand{\rmH}{\mathrm{H}}
\newcommand{\rmI}{\mathrm{I}}
\newcommand{\rmJ}{\mathrm{J}}
\newcommand{\rmK}{\mathrm{K}}
\newcommand{\rmN}{\mathrm{N}}
\newcommand{\rmO}{\mathrm{O}}

\newcommand{\rmQ}{\mathrm{Q}}

\newcommand{\rmX}{\mathrm{X}}

\newcommand{\rmZ}{\mathrm{Z}}

\newcommand{\bfC}{\mathbf{C}}

\renewcommand{\mod}{\mathbf{mod}}

\newcommand{\ab}{\mathbf{ab}}
\newcommand{\Ab}{\mathbf{Ab}}

\renewcommand{\phi}{\varphi}

\newcommand{\CAT}{\text{\upshape CAT}}
\newcommand{\TOP}{\text{\upshape TOP}}
\newcommand{\PL}{\text{\upshape PL}}
\newcommand{\DIFF}{\text{\upshape DIFF}}

\newcommand{\st}{\text{\upshape st}}
\newcommand{\op}{\text{\upshape op}}

\DeclareMathOperator{\Id}{Id}
\DeclareMathOperator{\pr}{pr}

\DeclareMathOperator*{\colim}{colim}

\DeclareMathOperator{\GL}{GL}
\DeclareMathOperator{\BGL}{BGL}
\DeclareMathOperator{\Inn}{Inn}
\DeclareMathOperator{\Out}{Out}
\DeclareMathOperator{\Aut}{Aut}
\DeclareMathOperator{\IA}{IA}
\DeclareMathOperator{\IO}{IO}

\DeclareMathOperator{\Hom}{Hom}

\DeclareMathOperator{\Tor}{Tor}
\DeclareMathOperator{\HH}{HH}
\DeclareMathOperator{\HML}{HML}
\DeclareMathOperator{\Lie}{Lie}

\newtheorem{theorem}{Theorem}
\newtheorem{proposition}[theorem]{Proposition}
\newtheorem{corollary}[theorem]{Corollary}

\theoremstyle{definition}
\newtheorem{definition}[theorem]{Definition}

\newtheorem{remark}[theorem]{Remark}

\numberwithin{theorem}{section}
\numberwithin{equation}{section}
\numberwithin{figure}{section}

\title{The rational stable homology\\ of mapping class groups\\ of universal nil-manifolds}

\author{Markus Szymik}

\newdateformat{mydate}{\monthname~\twodigit{\THEYEAR}}
\date{\mydate\today}

\begin{document}

\maketitle

\renewcommand{\abstractname}{\vspace{-2\baselineskip}}

\begin{abstract}
\noindent 
We compute the rational stable homology of the automorphism groups of free nilpotent groups. These groups interpolate between the general linear groups over the ring of integers and the automorphism groups of free groups, and we employ functor homology to reduce to the abelian case. As an application, we also compute the rational stable homology of the outer automorphism groups and of the mapping class groups of the associated aspherical nil-manifolds in the~$\TOP$,~$\PL$, and~$\DIFF$ categories.

\vspace{\baselineskip}

\vspace{\baselineskip}
\end{abstract}


\section{Introduction}

For any integer~$r\geqslant 1$, let~$\rmF_r$ denote the free group on~$r$ generators. The automorphism groups~$\Aut(\rmF_r)$ of these groups have been the subject of plenteous research. Nielsen began his investigations in that direction about a century ago, and more recent work has led to substantial progress in the computation of their homology. This reckoning applies in particular after stabilization: We can extend an automorphism of~$\rmF_r$ to~$\rmF_{r+1}$ by sending the new generator to itself, and in the~(co)limit~\hbox{$r\to\infty$}, we get the stable automorphism group~$\Aut(\rmF_\infty)$. Note that this notation does {\em not} refer to an automorphism group of a group~$\rmF_\infty$. The homo\-logy of the group~$\Aut(\rmF_\infty)$ is the colimit of the homo\-logies of the groups~$\Aut(\rmF_r)$. This stable homology has been computed by Galatius~\cite{Galatius}. In particular, it is rationally trivial, as had been conjectured by Hatcher and Vogtmann~\cite{Hatcher+Vogtmann:rational}. 

The descending sequence of normal subgroups  of the free group~$\rmF_r$ that is given by~$\Gamma_1(\rmF_r)=\rmF_r$ and~$\Gamma_{c+1}(\rmF_r)=[\,\rmF_r,\Gamma_c(\rmF_r)\,]$ is the lower central series. The quotient groups~\hbox{$\rmN^r_c=\rmF_r/\Gamma_{c+1}(\rmF_r)$} are the free nilpotent groups of class~$c\geqslant1$. 

For instance, the first case gives the free abelian groups~$\rmN^r_1\cong\bbZ^r$. Their automorphism groups are the general linear groups~$\GL_r(\bbZ)$ over the ring~$\bbZ$ of integers, and their homology is related to the algebraic K-theory of the ring~$\bbZ$. In contrast to the case of the groups~$\Aut(\rmF_r)$, we only know very few of their stable homology groups integrally. This changes when we are willing to work rationally. In that case, a complete computation of the stable homology has been achieved by Borel~\cite{Borel} using analytic methods. There is an isomorphism
\begin{equation}\label{eq:Borel}
\rmH^*(\GL_\infty(\bbZ);\bbQ)\cong\Lambda_\bbQ(r_5,r_9,r_{13},\dots)
\end{equation}
between the rational stable cohomology ring and a rational exterior algebra on generators in degrees~$5, 9, 13,\dots,4n+1,\dots$. The methods that we will use here are entirely algebraic and topological. They only provide new information relative to the homology of the general linear groups, to which we have nothing to add. We could, in principle, also attempt to use them to obtain information on torsion in the relative homology. At present, such computations are out of reach, however. 


In this writing, we address the general (and therefore non-abelian) free nilpotent groups~$\rmN^r_c$ of class~\hbox{$c\geqslant2$}~and rank~$r\geqslant2$. Among these is the famous Heisenberg group of unit upper-triangular~$(3,3)$--matrices with integral entries as the smallest example~$\rmN_2^2$. Its classifying space is the~$3$--dimensional Heisenberg manifold~$\rmX^2_2$. In general, the groups~$\rmN^r_c$ arise as fundamental groups of certain aspherical nil-manifolds~$\rmX^r_c$ that we can present as iterated torus bundles over tori. 

We are interested in the symmetries of these groups and manifolds, beginning with the automorphism groups~$\Aut(\rmN^r_c)$ of the groups~$\rmN^r_c$, and their homology. Again, we can stabilize by passing from~$r$ to~$r+1$, extending an automorphism by sending the new generator to itself. In the~(co)limit~$r\to\infty$, we get the stable automorphism group~$\Aut(\rmN_c^\infty)$. Again, this notation is {\em not} used to refer to the automorphism group of a group~$\rmN_c^\infty$. The homology of the group~$\Aut(\rmN_c^\infty)$ is the colimit of the homologies of the groups~$\Aut(\rmN_c^r)$. 


The homology of a group with constant coefficients is arguably the most fundamental case to consider. However, when we are dealing with families of groups, it turns out that it is often helpful to set up the computations to function in the much broader context of {\it polynomial}~(or {\it finite degree}) coefficients. See~\cite{Betley}, \cite{Djament}, and~\cite{Szymik}, for instance, and the proofs of Theorem~\ref{thm:Nil_c} and Theorem~\ref{thm:Out} below. The more comprehensive setting often allows to feed inductions, and whenever possible, it is, therefore, appropriate to give the results in that generality.

\begin{theorem}\label{thm:Nil_c}
Let~$V$ be a polynomial functor from the category~$\ab$ of finitely generated free abelian groups to the category of rational vector spaces. For every integer~$c\geqslant 1$ the canonical homomorphism
\[
\rmH_*(\Aut(\rmN_c^\infty);V_\infty)\longrightarrow\rmH_*(\GL_\infty(\bbZ);V_\infty)
\]
is an isomorphism. 
\end{theorem}

By~$2$-out-of-$3$, it follows that for every integer~$c\geqslant 2$ the canonical homomorphism
\[
\rmH_*(\Aut(\rmN_{c}^\infty);V_\infty)\longrightarrow\rmH_*(\Aut(\rmN_{c-1}^\infty);V_\infty)
\]
is an isomorphism. 

In the case of the constant functor~$V$ with~$V(\bbZ^r)=\bbQ$, we see that 
the natural homomorphism~$\Aut(\rmN_c^\infty)\to\GL_\infty(\bbZ)$ is a rational homology equivalence for every integer~$c\geqslant 1$: 

\begin{corollary}\label{cor:intro}
For all integers~$c\geqslant 2$ the canonical homomorphisms
\[
\cdots\longrightarrow\Aut(\rmN^r_c)\longrightarrow\Aut(\rmN^r_{c-1})
\longrightarrow\cdots
\longrightarrow\GL_r(\bbZ)
\]
induce isomorphisms in rational stable homology.
\end{corollary}


As a geometric application, we consider the automorphism groups (in various geometric categories) of the aspherical nil-manifolds~$\rmX^r_c$ with fundamental groups isomorphic to~$\rmN^r_c$. 

The group~$\pi_0\rmG(\rmX^r_c)$ of components of the topological monoid~$\rmG(\rmX^r_c)$ of homotopy self-equivalences of the classifying space~$\rmX^r_c$ of the group~$\rmN^r_c$ is isomorphic to the group~$\Out(\rmN^r_c)$ of outer automorphism of the group~$\rmN^r_c$. The following result is Theorem~\ref{thm:Out_new} in the main text.

\begin{theorem}\label{thm:Out}
For every class~$c\geqslant1$ the canonical projection induces an isomorphism
\[
\rmH_d(\Aut(\rmN^r_c);V(\bbZ^r))\longrightarrow\rmH_d(\Out(\rmN^r_c);V(\bbZ^r))
\]
in the stable range of~$\Aut(\rmN^r_c)$ for all finite degree functors~$V$ on~$\ab$ with values in rational vector spaces.
\end{theorem}

In particular, for all integers~$c\geqslant 1$ the canonical homomorphisms
\[
\Aut(\rmN^r_c)\rightarrow\Out(\rmN^r_c)
\]
of groups induce isomorphisms in rational stable homology. 


The following result is Theorem~\ref{thm:mcgs} in the main text. Starting from the isomorphisms~\hbox{$\Out(\rmN^r_c)\cong\pi_0\rmG(\rmX^r_c)$} of groups, it allows us to extend the preceding results from the homotopy category to the mapping class groups (groups of components of the automorphism groups) in the~$\TOP$,~$\PL$, and~$\DIFF$ categories.

\begin{theorem}\label{thm:TOP_PL_DIFF}
For all integers~$c\geqslant1$, if the nil-manifold~$\rmX^r_c$ is of dimension at least~$5$, the canonical homomorphisms
\[
\pi_0\DIFF(\rmX^r_c)\longrightarrow
\pi_0\PL(\rmX^r_c)\longrightarrow
\pi_0\TOP(\rmX^r_c)\longrightarrow
\pi_0\rmG(\rmX^r_c)
\]
of discrete groups induce isomorphisms in rational homology.
\end{theorem}

Taken together, Theorem~\ref{thm:Nil_c}, Theorem~\ref{thm:Out}, and Theorem~\ref{thm:TOP_PL_DIFF} compute the rational stable homology of the mapping class groups of the universal nil-manifolds~$\rmX^r_c$, justifying the title.


Stable homology computations such as those in the present paper produce most impact in the presence of homological stability: This guarantees that the stable homology determines infinitely many unstable values. Luckily, such results are already in place for the groups of interest here: Rational homological stability for the general linear groups over the ring~$\bbZ$ of integers is due to Borel~\cite{Borel}. With integral coefficients, the result is due to Charney~\cite{Cha80}, Maazen (unpublished), van der Kallen~\cite{vdK80}, and Suslin~\cite{Sus82}. Homological stability results for the automorphism groups of free groups are more recent than those for the general linear groups, see~\cite{Hat95} and~\cite{Hatcher+Vogtmann:Cerf}. They have been revisited by Bestvina~\cite{Bestvina} lately. For the automorphism groups of free nilpotent groups of an arbitrary class~$c$, the author has proven homological stability in~\cite{Szymik}. In none of theses cases is the exact stable range known, and we will not reproduce the estimates from the references. Our goal is the computation of the stable homology, and the mere existence of bounds for homological stability is enough to reach it.

An outline of this paper is as follows. Section~\ref{sec:homalg} contains the necessary background from homological algebra in functor categories and
provides for a basic rational vanishing result, Proposition~\ref{prop:rational_vanishing_c}. Section~\ref{sec:stable_K-theory} is about stable K-theory and recalls Scorichenko's theorem. In Section~\ref{sec:groups} we review the free nilpotent groups and their automorphisms. Section~\ref{sec:proofs} provides a proof of Theorem~\ref{thm:Nil_c}. Sections~\ref{sec:out} and~\ref{sec:mcg} contain the applications to the outer automorphism groups and the~$\TOP$,~$\PL$, and~$\DIFF$ mapping class groups of the associated aspherical nil-manifolds, respectively.


\section{Functor homology}\label{sec:homalg}

In this section, we briefly review some homological algebra in functor (and bifunctor) categories to the extent that we will need in the following sections. We also prove Proposition~\ref{prop:rational_vanishing_c}, a rational vanishing result that enters fundamentally into the proof of our main result.

\subsection{Mac Lane homology}

If~$\bfC$ is a small (or essentially small) category, then we write~$\calF(\bfC)$ for the category of functors~$\bfC\to\Ab$ from the category~$\bfC$ to the category~$\Ab$ of all abelian groups. This is an abelian category with enough projective objects to do homological algebra in. Note that we are looking at {\it all} functors, even if the category~$\bfC$ is such that it would make sense for us to consider only additive functors~$\bfC\to\Ab$. For instance, we will be particularly interested in the case when~$\bfC=\ab$ is the essentially small category~$\ab$ of finitely generated free abelian groups. Then we have the inclusion functor~$\rmI\colon\ab\to\Ab$ at our disposal. 

A pair of functors~$E\colon\bfC^\op\to\Ab$ and~$F\colon\bfC\to\Ab$ has an external tensor product, defined by
\begin{equation}\label{eq:external}
(E\boxtimes F)(X,Y)=E(X)\otimes F(Y),
\end{equation}
which gives a bifunctor~$E\boxtimes F\colon\bfC^\op\times\bfC\to\Ab$. Of course, not every bifunctor has this form. But every bifunctor $\bfC^\op\times\bfC\to\Ab$ has a coend, and in the case of~\hbox{$E\boxtimes F$} this is the internal tensor product $E\otimes F$, or sometimes $E\otimes_\bfC F$ for emphasis. This is just an abelian group. The derived functors will be denoted by~$\Tor_*^{\calF(\bfC)}$ or just~$\Tor_*^{\bfC}$. 

The following definition is due to Jibladze and Pirashvili, compare~\cite[Def.~1.2]{Jibladze+Pirashvili}.

\begin{definition}\label{def:HML}
The {\em Mac Lane homology} of the ring~$\bbZ$ of integers with coefficients in a functor~$F\in\calF(\ab)$ is defined as
\[
\HML_*(\bbZ;F)=\Tor_*^{\calF(\ab)}(\rmI^\vee,F),
\]
where~$\rmI^\vee\colon\ab^\op\to\Ab$ is the dual of the inclusion~$\rmI$, so that~$\rmI^\vee(X)=\Hom(X,\bbZ)$.
\end{definition}

\begin{remark}
A functor~$F\in\calF(\ab)$ is additive if and only if it is given (up to isomorphism) by tensoring  with an abelian group~\cite[p.~254]{Jibladze+Pirashvili}. In that case, the Mac Lane homology is the Hochschild homology of the Q-construction~$\rmQ(\bbZ)$ with coefficients in that abelian group, thought of as a~$(\bbZ,\bbZ)$--bimodule, see~\cite[2.1 and~6.2]{Jibladze+Pirashvili}. Also, in that case, the Mac Lane homology agrees with the topological Hochschild homology~(THH) of the Eilenberg--Mac Lane spectrum~$\rmH\bbZ$ with coefficients in the associated bimodule spectrum, see~\cite[Thm.~3.2]{Pirashvili+Waldhausen} and~\cite[Thm.~2.10]{FPSVW}.
\end{remark}

If~$\bbZ$ also denotes the constant functor~\hbox{$\bfC^\op\to\Ab$} with value the group~$\bbZ$ of integers, and~\hbox{$F\colon\bfC\to\Ab$} is any functor, the homology of the category $\bfC$ with coefficients in the functor $F$ is defined as
\begin{equation}\label{def:homology}
\rmH_*(\bfC;F)=\Tor_*^{\calF(\bfC)}(\bbZ,F).
\end{equation}

Both the Mac Lane homology as well as the homology of categories are special cases of the Hoch\-schild~(bifunctor) homology, as we shall now see.


\subsection{Hochschild (bifunctor) homology}

Let~$D\colon\bfC^\op\times\bfC\to\Ab$ be a bifunctor. For instance, we might have $D=E\boxtimes F$ as in~\eqref{eq:external}. Or we might have $D=\rmJ$, where $\rmJ(X,Y)=\bbZ\bfC(X,Y)$ is the free abelian group on the set $\bfC(X,Y)$ of morphisms $X\to Y$ in $\bfC$. The bifunctor $\rmJ$ is useful even when $\bfC$ is already enriched in abelian groups: The Hochschild homology of the category~$\bfC$ with coefficients in the bifunctor~$D$ is defined as
\[
\HH_*(\bfC;D)=\Tor_*^{\calF(\bfC^\op\times\bfC)}(\rmJ,D).
\]
See Loday~\cite[App.~C.10]{Loday} and Franjou--Pirashvili~\cite[2.5]{Franjou+Pirashvili:Scorichenko}. They use the notation~$\rmH$ for this, but we shall reserve it for the homology of categories with coefficients in a plain functor (rather than a bifunctor) as in~\eqref{def:homology}. The latter is a special case:

\begin{proposition}\label{prop:H_identification}
For all functors~$F\in\calF(\bfC)$ there is an isomorphism
\begin{equation}\label{eq:proj_iso}
\rmH_*(\bfC;F)\cong\HH_*(\bfC;F\circ\pr_2),
\end{equation}
where~$\pr_2\colon\bfC^\op\times\bfC\to\bfC$ is the projection.
\end{proposition}

\begin{proof}
Note that we can write the composition of functors as~\hbox{$F\circ\pr_2=\bbZ\boxtimes F$}, so that both sides of~\eqref{eq:proj_iso} are isomorphic to~$\Tor_*^{\calF(\bfC)}(\bbZ,F)$ by~\cite[Prop.~2.10 on p.~115]{Franjou+Pirashvili:Scorichenko} which says~\hbox{$\Tor_*^{\calF(\bfC)}(E,F)\cong\HH_*(\bfC;E\boxtimes F)$} in our notation, when $E$ takes values in free abelian groups.
\end{proof}

\begin{proposition}\label{prop:HML_identification}
For all functors~$F\in\calF(\ab)$ there is an isomorphism
\[
\HML_*(\bbZ;F)\cong\HH_*(\ab;\rmI^\vee\boxtimes F),
\]
where~$\rmI^\vee\boxtimes F$ is the bifunctor with~$(\rmI^\vee\boxtimes F)(X,Y)\cong\Hom(X,FY)$. 
\end{proposition}

\begin{proof}
Both sides are isomorphic to~$\Tor_*^{\calF(\ab)}(\rmI^\vee,F)$: the left hand side by Definition~\ref{def:HML}, and the right hand side by~\cite[Prop.~2.10 on p.~115]{Franjou+Pirashvili:Scorichenko} again. Alternatively, see~\cite[13.1.7, 13.2.17]{Loday}.
\end{proof}


\subsection{A rational vanishing result}

The following rational vanishing result will be the main computational input in our proof of Theorem~\ref{thm:Nil_c}. It refers to the Lie functors~$\Lie^b\colon\ab\to\Ab$, that is the degree~$b$ homogeneous part of the free Lie algebra functor. We will abbreviate
\[
\Lie^{[a,c]}=\bigoplus_{b=a}^c\Lie^b,
\]
so that~$\Lie^{[1,c]}(\bbZ^r)$ is nothing but the free nilpotent Lie algebra of class~$c$ on~$r$ generators.

\begin{proposition}\label{prop:rational_vanishing_c}
For all~$c\geqslant 2$ and~$q\geqslant1$, the Hochschild homology groups
\[
\HH_*(\ab;(\rmI^\vee\boxtimes\Lie^{[2,c]})^{\otimes q}\otimes V)
\]
vanish whenever $V$ is a functor on $\ab$ with values in rational vector spaces.
\end{proposition}

\begin{proof}
We will assume that the functor~$V$ is constant, so that we are dealing with the groups~$\HH_*(\ab;(\rmI^\vee\boxtimes\Lie^{[2,c]})^{\otimes q})$ rationally. The same argument gives the result in the general case, inserting~$\otimes V$ in suitable places.

We can check the first case~$q=1$ directly, because
\[
\HH_*(\ab;\rmI^\vee\boxtimes\Lie^{[2,c]})\cong\HML_*(\bbZ;\Lie^{[2,c]})
\]
is isomorphic to the Mac Lane homology of the ring~$\bbZ$ of integers with coefficients in the functor~$\Lie^{[2,c]}$ by Proposition~\ref{prop:HML_identification}. The Lie functors~$\Lie^b$ are rationally retracts of the tensor powers~$\bigotimes^b$: There are morphisms
\[
\Lie^b\longrightarrow\textstyle{\bigotimes^b}\longrightarrow\Lie^b
\]
such that their composition is multiplication with~$b$~\cite[Prop.~3.3]{Schlesinger}. Since the functor~$\bigotimes^b$ has vanishing Mac Lane homology by Pirashvili's Lemma~(compare~\cite[Lem.~2.2]{Franjou+Pirashvili}), so have the functors~$\Lie^b$, and therefore also their direct sum
\[
\Lie^{[2,c]}=\bigoplus_{b=2}^c\Lie^b.
\]
This proves the claim in the case~$q=1$.

To finish the proof, we reduce the case~$q\geqslant2$ to the previous one: There is an obvious isomorphism
\[
\big(\rmI^\vee\boxtimes\Lie^{[2,c]}\big)^{\otimes q}\cong
\big(\rmI^\vee\big)^{\otimes q}\boxtimes\big(\Lie^{[2,c]}\big)^{\otimes q},
\]
and by another invocation of~\cite[Prop.~2.10 on p.~115]{Franjou+Pirashvili:Scorichenko} we have
\[
\HH_*(\ab;(\rmI^\vee)^{\otimes q}\boxtimes(\Lie^{[2,c]})^{\otimes q})\cong\Tor_*^{\calF(\ab)}((\rmI^\vee)^{\otimes q},(\Lie^{[2,c]})^{\otimes q}).
\]
The right hand side can be computed, using the K\"unneth theorem for $\Tor$, from
\[
\Tor_*^{\calF(\ab)}(\rmI^\vee,\Lie^{[2,c]})^{\otimes q},
\]
and it therefore suffices to show that this is zero. But, working backwards, we already know that this is isomorphic to
\[
\HH_*(\ab;\rmI^\vee\boxtimes\Lie^{[2,c]})^{\otimes q}\cong\HML_*(\bbZ;\Lie^{[2,c]})^{\otimes q},
\]
and this vanishes as explained for $q=1$.
\end{proof}

\begin{remark}\label{rem:FP}
In the case~$c=2$ and~$q=1$ of the preceding result, the Mac Lane cohomology of the ring of integers with coefficients in the functor~$\Lie^2=\Lambda^2$ has been computed integrally by Franjou and Pirashvili~\cite[Cor.~2.3]{Franjou+Pirashvili}. It is all~$2$--torsion.
\end{remark}


\section{Stable K-theory}\label{sec:stable_K-theory}

There is a simple trick (due to Waldhausen, compare~\cite[Sec.~6]{Waldhausen:topII}) that helps us compute the stable homology
\begin{equation}\label{eq:coeffs}
\rmH_*(\GL_\infty(\bbZ);D_\infty)=\colim_r\rmH_*(\GL_r(\bbZ);D(\bbZ^r,\bbZ^r))
\end{equation}
of the general linear groups with coefficients in a module~$D_\infty$ that comes from a bifunctor~\hbox{$D\colon\ab\times\ab^\op\to\Ab$}: We can use the Serre spectral sequence for the fibration
\[
X\longrightarrow
\BGL_\infty(\bbZ)\longrightarrow
\BGL_\infty(\bbZ)^+,
\]
where the space~$X$ is the homotopy fiber of the indicated Quillen plus construction. The groups on the~$\rmE^2$ page take the form~$\rmH_*(\BGL_\infty(\bbZ)^+;\rmH_*(X;D_\infty))$, where the coefficients~$\rmH_*(X;D_\infty)$ are the homology of the fiber~$X$ with coefficients in the~$\pi_1(X)$--module~\hbox{$D_\infty=\colim_rD(\bbZ^r,\bbZ^r)$}.
The fundamental group~$\pi_1(X)$ acts on~$D_\infty$ via the homomorphism
$\pi_1(X)\to\pi_1(\BGL_\infty(\bbZ))\cong\GL_\infty(\bbZ)$.

\begin{remark}
The action of the fundamental group of~$\BGL_\infty(\bbZ)^+$ on these coefficients~$\rmH_*(X;D_\infty)$ is now trivial, see~\cite[Thm.~3.1]{Kassel:stab} and \cite[Pf.~of Thm.~2.16]{Kassel:calcul}. We will not need this fact, since we will show that in our situation the coefficients itself will be trivial.
\end{remark}

Since the Quillen plus construction is an acyclic map, it induces homology isomorphisms with respect to all coefficients, twisted and untwisted. Therefore, we can remove the plus to get a spectral sequence
\begin{equation}\label{eq:ss_Waldhausen_1}
\rmE^2_{s,t}\Longrightarrow\rmH_{s+t}(\GL_\infty(\bbZ);D_\infty),
\end{equation}
with
\begin{equation}\label{eq:ss_Waldhausen_2}
\rmE^2_{s,t}=\rmH_s(\GL_\infty(\bbZ);\rmK^\st_t(\bbZ;D))
\end{equation}
where, by definition,
\begin{equation}\label{eq:defKst}
\rmK^\st_*(\bbZ;D)=\rmH_*(X;D_\infty)
\end{equation}
is the {\em stable K-theory} of the ring of integers with coefficients in the bifunctor~$D$. 

\begin{remark}
B\"okstedt pointed out that this spectral sequence often degenerates. See also~\cite[Pf.~of Thm.~2]{Betley+Pirashvili}. Again, this will be obvious in our situation, because the entire~$\rmE^2$ page will be trivial.
\end{remark}

The definition~\eqref{eq:defKst} of stable K-theory is topological in nature. We will need the following result that gives an algebraic description of it. 

\begin{theorem}[{\bf Scorichenko}]\label{thm:Scorichenko}
For any ring~$R$, let~$\mod_R$ be the category of finitely generated projective left~$R$--modules, and let~$D$ be a bifunctor on it. If~$D$ has finite degree with respect to each of the variables, then there is an isomorphism between Waldhausen's stable~K-theory and the Hochschild homology of the category~$\mod_R$:
\[
\rmK^\st_*(R;D)\cong\HH_*(\mod_R;D).
\] 
\end{theorem}

See the exposition~\cite{Franjou+Pirashvili:Scorichenko} by Franjou and Pirashvili. The cited result is stated as Theorem~1.1 there, and the proof given is complete for rings~$R$ with the property that submodules of finitely generated projective left~$R$--modules are still finitely generated and projective. This property is satisfied for~$R=\bbZ$, when~$\mod_\bbZ=\ab$, which is the only case that we will be using here. See also~\cite[5.2]{Djament} for an enlightening discussion, and~\cite[App.]{FFSS}.


\section{Free nilpotent groups and their automorphisms}\label{sec:groups}

In this section, we present some basic results about the free nilpotent groups and their automorphisms. Most of this must be well-known, and we can refer to the exposition in~\cite{Szymik}, for instance. We will, however, prove the two Propositions~\ref{prop:homology_is_polynomial} and~\ref{prop:direct_summand} which will later be used.

Let~$G$ be a (discrete) group. For integers~$n\geqslant1$, the subgroups~$\Gamma_n(G)$ are defined inductively by~$\Gamma_1(G)=G$ and~$\Gamma_{n+1}(G)=[G,\Gamma_n(G)]$. We also set~$\Gamma_\infty(G)$ to be the intersection of all the~$\Gamma_n(G)$. This gives a series
\[
G=\Gamma_1(G)\geqslant\Gamma_2(G)\geqslant\dots\geqslant\Gamma_\infty(G)
\]
of normal subgroups, the {\em descending/lower central series} of~$G$. The associated graded group is abelian, and the commutator bracket induces the structure of a graded Lie algebra~$\Lie(G)$ on it. 

Let~$\rmF_r$ denote a free group on a set of~$r$ generators. In this case, an old theorem of Magnus says that the subgroup~$\Gamma_\infty(\rmF_r)$ is trivial, so that the free groups~$\rmF_r$ are {\em residually nilpotent}. By Witt's theorem, the canonical homomorphism from a free Lie algebra on a set of~$r$ generators to the associated graded Lie algebra of~$\rmF_r$ is an isomorphism. In particular 
\begin{equation}\label{eq:Witt2}
\Gamma_n(\rmF_r)/\Gamma_{n+1}(\rmF_r)\cong\Lie^n(\bbZ^r)
\end{equation}
is a free abelian group, the degree~$n$ homogeneous part of the free Lie algebra on a set of~$r$ generators. For instance, we have~$\Lie^1=\Id$ and~$\Lie^2=\Lambda^2$.


\subsection{Free nilpotent groups}

The universal examples of nilpotent groups of class~$c\geqslant1$ are the quotients
\[
\rmN_c^r=\rmF_r/\Gamma_{c+1}(\rmF_r).
\] 
As the two extreme cases, we obtain~$\rmN^r_1\cong\bbZ^r$ and~$\rmN^r_\infty\cong\rmF_r$. It follows from~\eqref{eq:Witt2} that there are extensions
\begin{equation}\label{eq:Nextension}
0\longrightarrow\Lie^c(\bbZ^r)\longrightarrow\rmN_c^r\longrightarrow\rmN_{c-1}^r\longrightarrow1
\end{equation}
of groups.


\subsection{Rational homology}

We will later need the following structural result on the rational homology of the groups~$\rmN_c^r$.

\begin{proposition}\label{prop:homology_is_polynomial}
For every class~$c\geqslant1$ and every degree~$d$, there is a polynomial functor~$P$, depending on~$c$ and~$d$, and of degree at most~$cd$, from~$\ab$ to rational vector spaces, such that~$\rmH_d(\rmN_c^r;\bbQ)\cong P(\bbZ^r)$ for all~$r$.
\end{proposition}

\begin{proof}
Since we are only interested in the rational homology of nilpotent groups, we may just as well consider the rationalizations/Malcev completions~\hbox{$\rmN_c^r\otimes\bbQ$}, see~\cite[Sec.~4]{Lazard}. We have natural isomorphisms~$\rmH_d(\rmN_c^r;\bbQ)\cong\rmH_d(\rmN_c^r\otimes\bbQ;\bbQ)$. Then, the category of uniquely divisible nilpotent groups is equivalent to the category of nilpotent Lie algebras over~$\bbQ$, and we get a natural isomorphism
\[
\rmH_d(\rmN_c^r;\bbQ)\cong\rmH_d(\Lie(\rmN_c^r\otimes\bbQ);\bbQ),
\]
as demonstrated by Nomizu~\cite[Thm.~1]{Nomizu} and, more generally and explicitly, by Pickel~\cite[Thm.~10]{Pickel}. As a consequence of Witt's theorem, the Lie algebra~\hbox{$\Lie(\rmN_c^r\otimes\bbQ)$} of the group of interest is nothing else than the free nilpotent Lie algebra~$\Lie^{[1,c]}(\bbQ^r)$ on~$r$ generators, and this is clearly functorial in~$\bbQ^r$ (and in particular in~$\bbZ^r$) of degree~$c$. The~$d$--th homology of a Lie algebra~$\mathfrak{g}$ is the homology of the Chevalley--Eilenberg resolution
\[
\cdots\longleftarrow\Lambda^d\mathfrak{g}\longleftarrow\cdots.
\]
As a subquotient of the functor~$\Lambda^d\Lie(\rmN_c^r\otimes\bbQ)$, which is of degree~$cd$, the homology is then also functorial of degree at most~$cd$.
\end{proof}


\subsection{Automorphisms}

More than in the free nilpotent groups~$\rmN_c^r$ themselves, we are interested in the groups~$\Aut(\rmN_c^r)$ of their automorphisms. In the case~$c=1$, these are the general linear groups~$\GL_r(\bbZ)$, and in the limiting case~$c=\infty$, these are the automorphism groups~$\Aut(\rmF_r)$ of free groups. 

\begin{proposition}
There are extensions
\begin{equation}\label{prop:aut_extension}
0
\longrightarrow\Hom(\bbZ^r,\Lie^c(\bbZ^r))
\longrightarrow\Aut(\rmN_c^r)
\longrightarrow\Aut(\rmN_{c-1}^r)
\longrightarrow1
\end{equation}
of groups.
\end{proposition}

See again~\cite[Prop.~3.1]{Szymik}, for instance. Note that any element of the quotient group~$\Aut(\rmN_{c-1}^r)$ acts on the kernel group via restriction along the canonical projection~$\Aut(\rmN_{c-1}^r)\to\Aut(\rmN_1^r)=\GL_r(\bbZ)$. For that reason, it is sometimes more efficient to study the extensions
\begin{equation}\label{eq:def_IA}
1
\longrightarrow\IA_c^r
\longrightarrow\Aut(\rmN_c^r)
\longrightarrow\GL_r(\bbZ)
\longrightarrow1
\end{equation}
of groups. The kernels~$\IA_c^r$ are known to be nilpotent of class~$c-1$. See for instance Andreadakis~\cite[Cor.~1.3.]{Andreadakis}. By what has already been said, it is clear that they differ by free abelian groups in the sense that there are extensions
\begin{equation}\label{eq:aut_extension_rel}
0
\longrightarrow\Hom(\bbZ^r,\Lie^c(\bbZ^r))
\longrightarrow\IA_c^r
\longrightarrow\IA_{c-1}^r
\longrightarrow1.
\end{equation}


\subsection{On the Andreadakis--Johnson module}

We can now put this together to get the following result that will later help us to estimate homology.

\begin{proposition}\label{prop:direct_summand}
The~$\GL_r(\bbZ)$--module~$\rmH_q(\IA_c^r)\otimes\bbQ$ is isomorphic to a direct summand of the~$\GL_r(\bbZ)$--module~\hbox{$\Lambda^q(\Hom(\bbZ^r,\Lie^{[2,c]}(\bbZ^r))\otimes\bbQ$}.
\end{proposition}

\begin{proof}
Our starting point are the Lyndon--Hochschild--Serre spectral sequences associated with the extensions~\eqref{eq:aut_extension_rel}. We have
\[
\rmE^2_{s,t}\Longrightarrow\rmH_{s+t}(\IA_c^r)
\]
with
\begin{equation}\label{eq:E2}
\rmE^2_{s,t}=\rmH_s(\IA_{c-1}^r;\rmH_t\Hom(\bbZ^r,\Lie^c(\bbZ^r))).
\end{equation}

Since we are working rationally, and~$\GL_r$ is reductive, every short exact sequence splits. It therefore suffices to show that the filtration quotients~$\rmE^\infty_{s,q-s}$ of the homology~$\rmH_q(\IA_c^r)$ form a direct summand of~\hbox{$\Lambda^q(\Hom(\bbZ^r,\Lie^{[2,c]}(\bbZ^r))$}. In turn, this will follow once we know the same for the terms~$\rmE^2_{s,q-s}$ on the~$\rmE^2$ page, because the groups~$\rmE^\infty_{s,q-s}$ are subquotients (and rationally direct summands) thereof.

As for the groups~\eqref{eq:E2} on the~$\rmE^2$ page, it is evident by construction that the action of the kernels~$\IA_{c-1}^r$ of the canonical projections~\hbox{$\Aut(\rmN_c^r)\to\GL_r(\bbZ)$} on the coefficient groups~\hbox{$\Hom(\bbZ^r,\Lie^c(\bbZ^r))$} are trivial. Therefore, we can write the~$\rmE^2$ term as a graded tensor product
\[
\rmE^2_{*,*}\cong\rmH_*(\IA_{c-1}^r)\otimes\Lambda^*\Hom(\bbZ^r,\Lie^c(\bbZ^r)),
\]
using that the homology of the free abelian groups is given by the exterior powers. We see that~$\rmH_q(\IA_c^r)$ is a direct summand of the degree~$q$ part of
\[
\rmH_*(\IA_{c-1}^r)\otimes\Lambda^*\Hom(\bbZ^r,\Lie^c(\bbZ^r)).
\]
By induction on~$c$, this is isomorphic to a direct summand of the degree~$q$ part of
\begin{align*}
&\Lambda^*\Hom(\bbZ^r,\Lie^2(\bbZ^r))
\otimes\cdots\otimes
\Lambda^*\Hom(\bbZ^r,\Lie^c(\bbZ^r))\\
\cong&
\Lambda^*\Hom(\bbZ^r,\Lie^2(\bbZ^r)
\oplus\cdots\oplus\Lie^c(\bbZ^r)),
\end{align*}
which is~$\Lambda^q\Hom(\bbZ^r,\Lie^{[2,c]}(\bbZ^r))$, as claimed.
\end{proof}

Tensoring with another rational representation~$V(\bbZ^r)$ that originates in a functor~$V$ from the category of finitely generated free abelian groups to the category of rational vector spaces, we obtain:

\begin{corollary}\label{cor:direct_summand}
For every functor~$V$ from the category~$\ab$ of the finitely generated free abelian groups to the category of rational vector spaces, the~$\GL_r(\bbZ)$--module~$\rmH_q(\IA_c^r)\otimes V(\bbZ^r)$ is a direct summand (up to isomorphism) of the~$\GL_r(\bbZ)$--module~\hbox{$\Lambda^q(\Hom(\bbZ^r,\Lie^{[2,c]}(\bbZ^r))\otimes V(\bbZ^r)$}.
\end{corollary}


\section{A proof of Theorem~\ref{thm:Nil_c}}\label{sec:proofs}

We now have enough ingredients to prove Theorem~\ref{thm:Nil_c}.

\begin{proof}
The homomorphism
\begin{equation}\label{eq:edge_c}
\rmH_*(\Aut(\rmN_c^\infty);V_\infty)\longrightarrow\rmH_*(\GL_\infty(\bbZ);V_\infty)
\end{equation}
in question is the edge homomorphism in the Lyndon--Hochschild--Serre spectral sequence
\[
\rmE^2_{p,q}\Longrightarrow\rmH_{p+q}(\Aut(\rmN_c^\infty);V_\infty)
\]
for the colimit of the extensions~\eqref{eq:def_IA}, with~$\rmE^2$ page
\begin{equation}\label{eq:first_ss_2}
\rmE_{p,q}^2=\rmH_p(\GL_\infty(\bbZ);\rmH_q(\IA_c^\infty;V_\infty)).
\end{equation}
Since~$\IA_c^\infty$ maps trivially to the general linear groups, we can pull the coefficients out:
\[
\rmE_{p,q}^2\cong\rmH_p(\GL_\infty(\bbZ);\rmH_q(\IA_c^\infty)\otimes V_\infty).
\]
In fact, the~$0$--line is just the homology of the stable general linear group with coefficients in~$V_\infty$. We will show that the other groups (i.e.~the lines with~$q\geqslant1$) on the~$\rmE^2$ page vanish rationally: 
\begin{equation}\label{eq:claim}
\rmH_*(\GL_\infty(\bbZ);\rmH_q(\IA_c^\infty)\otimes V_\infty)=0
\end{equation}
for all~$c\geqslant2$ and~$q\geqslant1$. This will imply the result.

We will prove~\eqref{eq:claim} using some very crude estimates. (Here these always originate from spectral sequences.) The first one is Corollary~\ref{cor:direct_summand}, which implies that it suffices to be shown that 
\begin{equation}
\rmH_*(\GL_\infty(\bbZ);\Lambda^q(\rmI^\vee\boxtimes\Lie^{[2,c]})_\infty\otimes V_\infty)=0
\end{equation}
for all~$c\geqslant2$ and~$q\geqslant1$. As explained in Section~\ref{sec:stable_K-theory}, such twisted homology can be computed from the stable~K-theory~$\rmK_*^\st(\bbZ;\Lambda^q(\rmI^\vee\boxtimes\Lie^{[2,c]})\otimes V)$ using the spectral sequence~\eqref{eq:ss_Waldhausen_1} and~\eqref{eq:ss_Waldhausen_2}
for the bifunctor~$D=\Lambda^q(\rmI^\vee\boxtimes\Lie^{[2,c]})\otimes V$:
\[
\rmE^2_{s,t}\Longrightarrow\rmH_{s+t}(\GL_\infty(\bbZ);\Lambda^q(\rmI^\vee\boxtimes\Lie^{[2,c]})_\infty\otimes V_\infty)
\]
with
\[
\rmE^2_{s,t}=\rmH_s(\GL_\infty(\bbZ);\rmK_t^\st(\bbZ;\Lambda^q(\rmI^\vee\boxtimes\Lie^{[2,c]})\otimes V)).
\]
The bifunctor~$\Lambda^q(\rmI^\vee\boxtimes\Lie^{[2,c]})\otimes V$ at hand is of finite degree in each variable, so that we can use Scorichenko's Theorem~\ref{thm:Scorichenko} to the effect that stable K-theory is bifunctor homology,
\[
\rmK_*^\st(\bbZ;\Lambda^q(\rmI^\vee\boxtimes\Lie^{[2,c]})\otimes V)
\cong\HH_*(\ab;\Lambda^q(\rmI^\vee\boxtimes\Lie^{[2,c]})\otimes V).
\]
Thus, it suffices to see that the right hand side is zero rationally. Since the functor~$\Lambda^q$ is a retract of~$\otimes^q$, this follows from Proposition~\ref{prop:rational_vanishing_c}.
\end{proof}

\begin{remark}\label{rem:vanish1}
By Betley's work~\cite[Thm.4.2]{Betley} on the homology of the general linear groups with twisted coefficients, the right hand side of~\eqref{eq:edge_c} (and therefore also the left hand side) vanishes if the functor~$V$ is reduced, that is if we have~$V(0)=0$.
\end{remark} 


\section{Outer automorphism groups}\label{sec:out}

The outer automorphism group~$\Out(G)$ of a group~$G$ is the quotient in an extension
\[
1\longrightarrow \Inn(G)\longrightarrow\Aut(G)\longrightarrow\Out(G)\longrightarrow1,
\]
with inner automorphism group~$\Inn(G)\cong G/\rmZ(G)$ and the projection gives a factorization~\hbox{$\Aut(G)\to\Out(G)\to\Aut(\rmH_*G)$} of the canonical homomorphism to the automorphism group of the homo\-logy/abelianization: Inner automorphisms act trivially on homology. 


In the case when~$G=\rmN_c^r$ is a free nilpotent group, the center is the kernel of the canonical projection~$\rmN_c^r\to\rmN_{c-1}^r$, so that we have
\[
\rmZ(\rmN_c^r)\cong\Lie^c(\bbZ^r)
\]
and
\[
\Inn(\rmN_c^r)\cong\rmN_{c-1}^r.
\]
Consequently, we get an extension
\begin{equation}\label{eq:aut_out_extension}
1\longrightarrow\rmN_{c-1}^r\longrightarrow\Aut(\rmN_c^r)\longrightarrow\Out(\rmN_c^r)\longrightarrow1.
\end{equation}
In parallel with~\eqref{eq:def_IA}, we define subgroups~$\IO_c^r\leqslant\Out(\rmN_c^r)$ as the kernels in the extensions
\begin{equation}\label{eq:def_IO}
1
\longrightarrow\IO_c^r
\longrightarrow\Out(\rmN_c^r)
\longrightarrow\GL_r(\bbZ)
\longrightarrow1.
\end{equation}
This gives extensions
\begin{equation}\label{eq:main_extension}
1\longrightarrow\rmN_{c-1}^r\longrightarrow\IA_c^r\longrightarrow\IO_c^r\longrightarrow1
\end{equation}
of nilpotent groups of class~$c-1$. 

This finishes our description of the outer automorphism groups~$\Out(\rmN_c^r)$ up to extensions. We can proceed to compute their rational homology.



Let~$V$ be a polynomial functor (of finite degree) from the category~$\ab$ of finitely generated free abelian groups to the category of rational vector spaces. This defines twisted coefficients for both of the groups~$\Aut(\rmN_c^\infty)$ and~$\Out(\rmN_c^\infty)$ via their homomorphisms to the general linear groups.

\begin{theorem}\label{thm:Out_new}
For every class~$c\geqslant1$ the canonical projection induces an isomorphism
\begin{equation}\label{eq:aut->out}
\rmH_d(\Aut(\rmN^r_c);V(\bbZ^r))\longrightarrow\rmH_d(\Out(\rmN^r_c);V(\bbZ^r))
\end{equation}
in the stable range of~$\Aut(\rmN^r_c)$ for all finite degree functors~$V$ on~$\ab$ with values in rational vector spaces.
\end{theorem}

\begin{remark}
The existence of a stable range for the homology of the automorphism groups (with polynomial coefficients) has been established in~\cite[Thm.~4.5]{Szymik}, even integrally. The proof of Theorem~\ref{thm:Out_new} given below uses that stability result. It says that there exists an integer~\hbox{$R=R(c,d,V)$} such that the homology~$\rmH_d(\Aut(\rmN^r_c);V(\bbZ^r))$ is independent (in a precise sense) of~$r$ as soon as~$r\geqslant R$. At the time of writing, the best possible value is not known. 
\end{remark} 

\begin{remark}\label{rem:vanish2}
The theorem implies the existence of a stable range for the homology of the {\it outer} automorphism groups as well, at least rationally. We note that, in particular, the rational homology vanishes in the stable range if the functor~$V$ is reduced, by Theorem~\ref{thm:Nil_c} and Remark~\ref{rem:vanish1} following it.
\end{remark} 

\begin{proof}[Proof of Theorem~\ref{thm:Out_new}]
We prove this by induction on the homological degree~$d$.

We start the induction in degree~$d=0$. This case concerns the co-invariants of the actions on~$V(\bbZ^r)$. Since the groups~$\Aut(\rmN^r_c)$ surject onto the groups~$\Out(\rmN^r_c)$, the co-invariants are the same.

Let us now assume that the degree~$d$ is positive, and that the result has already been proven for all smaller degrees than that. We consider the Lyndon--Hochschild--Serre spectral sequence 
\[
\rmE_{s,t}^2(r)=\rmH_s(\Out(\rmN^r_c);\rmH_t(\rmN^r_{c-1};V(\bbZ^r)))\Longrightarrow\rmH_{s+t}(\Aut(\rmN^r_c);V(\bbZ^r))
\]
for the extension~\eqref{eq:aut_out_extension} in the stable range, that is for~$r$ so large that the homology groups under consideration represent the stable values. (There are only finitely many of them relevant at each given time.) We would like to show that the rows with~$t\not=0$ vanish, so that the edge homomorphism--which is the homomorphism~\eqref{eq:aut->out} in question--is an isomorphism.

Since the action on the coefficients~$V(\bbZ^r)$ factors through the quotient~$\Out(\rmN^r_c)$, the coefficients for the homology~$\rmH_t(\rmN^r_{c-1};V(\bbZ^r))$ are actually untwisted. 
\[
\rmH_t(\rmN^r_{c-1};V(\bbZ^r))\cong\rmH_t(\rmN^r_{c-1})\otimes V(\bbZ^r)
\]
From Proposition~\ref{prop:homology_is_polynomial} we deduce that this
is a polynomial functor (of finite degree). And if~$t\not=0$, then this functor is reduced: Since the group~$\rmN^0_{c-1}$ is trivial, so is the homology~$\rmH_t(\rmN^0_{c-1})$ for~$t\not=0$. It follows by induction (and Remark~\ref{rem:vanish2}) that the groups~$\rmE_{s,t}^2(r)$ vanish for~$t\not=0$, and we are left with an isomorphism
\[
\rmH_d(\Aut(\rmN_c^r);V(\bbZ^r))\cong\rmE_{d,0}^2(r)=\rmH_d(\Out(\rmN_c^r);V(\bbZ^r)),
\]
as desired.
\end{proof}

\begin{remark}
The basic set-up for this proof is the same as in~\cite[Sec.~1]{Szymik}. There it was used to prove homological stability, whereas here it is used to compute stable homology.
\end{remark}

Combining the specializations of both Theorems~\ref{thm:Nil_c} and~\ref{thm:Out_new} to the case of constant coefficients we get:

\begin{corollary}\label{cor:Out_New}
The canonical projection induces an isomorphism
\[
\rmH_*(\Out(\rmN_c^\infty);\bbQ)\longrightarrow\rmH_*(\GL_\infty(\bbZ);\bbQ)
\]
in rational homology for all classes~$c\geqslant1$.
\end{corollary}


\section{Mapping class groups}\label{sec:mcg}

The free nilpotent groups~$\rmN^r_c$ of class~$c$ on~$r$ generators arise as fundamental groups of aspherical nil-manifolds: Upon passage to classifying spaces, the extensions~\eqref{eq:Nextension} gives rise to iterated torus bundles over tori. These manifolds will be denoted by~$\rmX^r_c$ here. We can now spell out some implications of our preceding results for the mapping class groups of the~$\rmX^r_c$ for the various categories of automorphisms. Note that the dimension~$\dim(\rmX^r_c)$ of the manifolds~$\rmX^r_c$ increases with~$r$, so that for `large' enough~$r$, which is the situation that we are interested in here, surgery and concordance/pseudo-isotopy theory apply and will compute the mapping class groups in the~$\TOP$,~$\PL$, and~$\DIFF$ categories. For basic information about the automorphism groups of manifolds in these categories, we refer to the survey articles by Burghelea~\cite{Burghelea}, Hatcher~\cite{Hatcher}, Balcerak--Hajduk~\cite{Balcerak+Hajduk}, and Weiss--Williams~\cite{Weiss+Williams}.

We begin with the homotopy category. For any discrete group~$G$, the group~$\Out(G)$ is isomorphic to the group of components of the topological monoid~$\rmG(X)$ of homotopy self-equivalences of the classifying space~$X$ of~$G$. For the groups~$G=\rmN^r_c$ with classifying spaces~$\rmX^r_c$, we therefore have an isomorphism
\begin{equation}\label{eq:G=Out}
\pi_0\rmG(\rmX^r_c)\cong\Out(\rmN^r_c)
\end{equation}
of groups, and Corollary~\ref{cor:Out_New} directly gives the stable rational homology of the symmetries of the manifolds~$\rmX^r_c$ in the homotopy category. 

We now explain how to transfer this to the~$\TOP$,~$\PL$, and~$\DIFF$ categories:

\begin{theorem}\label{thm:mcgs}
For all integers~$c\geqslant1$, if the nil-manifold~$\rmX^r_c$ is of dimension at least~$5$, the canonical homomorphisms
\[
\pi_0\DIFF(\rmX^r_c)\longrightarrow
\pi_0\PL(\rmX^r_c)\longrightarrow
\pi_0\TOP(\rmX^r_c)\longrightarrow
\pi_0\rmG(\rmX^r_c)
\]
of discrete groups induce isomorphisms in rational homology of discrete groups.
\end{theorem}

\begin{remark}
The first case~$c=1$ of this result is already known as a consequence of the~(independent) work of Hsiang--Sharpe~\cite{Hsiang+Sharpe} and Hatcher~\cite{Hatcher}. In order to generalize their arguments, we need to know the vanishing of surgery obstructions for manifolds with free nilpotent fundamental groups. Since the latter are poly--$\bbZ$ groups by~\eqref{eq:Witt2}, 
we can refer to Wall~\cite[15B]{Wall}. Waldhausen~\cite[Thms.~17.5, 19.4]{Waldhausen} has shown that the Whitehead spaces vanish for a class of groups that contains all poly--$\bbZ$ groups.
\end{remark}

\begin{proof}[Proof of Theorem~\ref{thm:mcgs}]
We start with the homeomorphisms. It is clear that the canonical homomorphism~$
\pi_0\TOP(\rmX^r_c)\to\pi_0\rmG(\rmX^r_c)$ is surjective. For~$\dim(\rmX^r_c)\geqslant 5$ this follows from surgery theory, see for instance~\cite[Thm.~15B.1]{Wall}. In our situation, it can also be seen explicitly, and for all~$\rmX^r_c$, because the projection~$\Aut(\rmN^r_c)\to\Out(\rmN^r_c)$ is evidently surjective, and~$\Aut(\rmN^r_c)$ tautologically acts on~$\rmN^r_c$, on the completion~$\rmN^r_c\otimes\bbR$ of~$\rmN^r_c\otimes\bbQ$, and therefore also on the quotient~$\rmX^r_c\cong\rmN^r_c\otimes\bbR/\rmN^r_c$. This action is algebraic, hence smooth.

We need to estimate the kernel of the homomorphism~$\pi_0\TOP(\rmX^r_c)\to\pi_0\rmG(\rmX^r_c)$, which is isomorphic to the fundamental group~$\pi_1(\rmG(\rmX^r_c)/\TOP(\rmX^r_c))$. It follows from surgery theory~(see again~\cite[Sec.~15B,~17A]{Wall} and the remark preceding this proof) that~$\rmG(\rmX^r_c)/\widetilde\TOP(\rmX^r_c)$ is contractible, where the notation~$\widetilde\TOP$ as usually refers to the block homeomorphism group. We thus have an isomorphism 
\[
\pi_1(\rmG(\rmX^r_c)/\TOP(\rmX^r_c))\cong\pi_1(\widetilde\TOP(\rmX^r_c)/\TOP(\rmX^r_c)).
\]
The right hand side can be determined from concordance/pseudo-isotopy theory as follows: Work of Hatcher and Igusa (\cite[Sec.~2,~3]{Hatcher} and~\cite[Sec.~8]{Igusa}) shows that this group is dominated by~$\pi_0$ of the concordance space of~$\rmX^r_c$. Therefore, the only possible contribution to~$\pi_0$ is a quotient of the~$2$--torsion group~$\bbF_{2}[\rmN^r_c]/\bbF_2$. We see that this vanishes rationally (in fact, away from~$2$), so that 
\[
\pi_0\TOP(\rmX^r_c)\longrightarrow\pi_0\rmG(\rmX^r_c)
\]
is a rational homology isomorphism if~$\dim(\rmX^r_c)\geqslant 5$. This settles the~$\TOP$ case.

Since the group~$\pi_1(\widetilde\CAT(\rmX^r_c)/\CAT(\rmX^r_c))$ is independent of the category~$\CAT$ of symmetries considered~\cite{Hatcher}, we see that it vanishes rationally for the remaining two cases~$\CAT=\PL$ and~$\CAT=\DIFF$ as well. For this reason, the kernels of the homomorphisms
\begin{equation}\label{eq:mcgs}
\pi_0\DIFF(\rmX^r_c)\longrightarrow
\pi_0\PL(\rmX^r_c)\longrightarrow
\pi_0\TOP(\rmX^r_c)
\end{equation}
between the mapping class groups in the various categories of structures are the same as the kernels between the corresponding blocked groups~$\pi_0\widetilde\CAT(\rmX^r_c)$, and it remains to be checked that these kernels are rationally trivial.

From surgery theory, see the references for~\cite[Sec.~17A]{Wall}, there exists an equivalence between the quotient space~$\widetilde\TOP(\rmX^r_c)/\widetilde\PL(\rmX^r_c)$, which measures the difference between the two block automorphism groups, and the function space of maps~\hbox{$\rmX^r_c\to\TOP/\PL$}. The work of Kirby--Siebenmann~(see~\cite[Rem.~2.25]{Madsen+Milgram} for instance) shows that the target space~$\TOP/\PL$ is an Eilen\-berg--Mac Lane space of type~$(\bbZ/2,3)$, which is rationally trivial. Therefore, the function space is also rationally trivial. 

We similarly have an equivalence between~$\widetilde\PL(\rmX^r_c)/\widetilde\DIFF(\rmX^r_c)$ and the function space of maps~\hbox{$\rmX^r_c\to\PL/\rmO$}. By work of Cerf and Kervaire--Milnor, the target~$\PL/\rmO$ also has finite homotopy groups, essentially the groups of exotic spheres, see~\cite[Rem.~4.21]{Madsen+Milgram}. We can therefore argue as in the paragraph before.
\end{proof}


\section*{Acknowledgments}

I thank A. Djament for his constructive comments on early drafts of this text. His and a referee's suggestions led to substantial improvements. I also thank B.I. Dundas, W.G. Dwyer, T. Pirashvili, A. Putman, C. Vespa, and N. Wahl.

This research has been supported by the Danish National Research Foundation through the Centre for Symmetry and Deformation~(DNRF92), and parts of this paper were written while I was visiting the Hausdorff Research Institute for Mathematics, Bonn. 



\vfill

\parbox{\linewidth}{%
Department of Mathematical Sciences\\
NTNU Norwegian University of Science and Technology\\
7491 Trondheim\\
NORWAY\\
\phantom{ }\\
\href{mailto:markus.szymik@ntnu.no}{markus.szymik@ntnu.no}\\
\href{https://folk.ntnu.no/markussz}{https://folk.ntnu.no/markussz}
}


\end{document}